\documentclass[reqno,11pt]{amsart}
\usepackage{amsmath,amsfonts,mathrsfs,amssymb,graphicx}

\newcommand{\Z}{{\mathbb Z}}
\newcommand{\R}{{\mathbb R}}
\newcommand{\Q}{{\mathbb Q}}

 \allowdisplaybreaks[4]
\newtheorem{thm}{Theorem}[section]

\newtheorem{prop}[thm]{Proposition}
\newtheorem{rem}[thm]{Remark}
\newtheorem{lem}[thm]{Lemma}

\begin{document}
\title{On the Hausdorff dimension of the spectrum of  Thue-Morse Hamiltonian }

\author{Qinghui LIU}
\address[Q.H. LIU]{
Dept. Computer Sci.,
Beijing Institute of Technology,
Beijing 100081, PR China.}
\email{qhliu@bit.edu.cn}
\author{Yanhui QU}
\address[Y.H. QU]{Dept. Math., Tsinghua University, Beijing 100084, PR China.}
\email{yhqu@math.tsinghua.edu.cn}

\begin{abstract}
We show that  the Hausdorff dimension of the spectrum  of the Thue-Morse Hamiltonian  has a common positive lower bound for all coupling.
\end{abstract}

\maketitle


\section{Introduction}

Given a bounded real sequence $v=\{v(n)\}_{n\in \Z}$ and a real number $\lambda\in \R$, we can define the so-called {\it discrete Schr\"odinger operator } $H_{\lambda,v}$ acting on $\ell^2(\Z)$ as
$$
(H_{\lambda,v}\psi)(n)=\psi(n+1)+\psi(n-1)+\lambda v(n)\psi(n), \ \ \ \forall n\in \Z.
$$
Here $\lambda$ is called the {\it coupling} constant and $\lambda v$ is called the {\it potential}. It is well known that $H_{\lambda,v}$ is a self-adjoint operator and the spectrum  of $H_{\lambda,v}$ is a compact subset of $\R$, which we denote by $\sigma(H_{\lambda,v})$(see for example  \cite{CL}). We concern about the size of $\sigma(H_{\lambda,v}).$ For example  whether it has positive Lebesgue measure? If not, what is the Hausdorff dimension of it?
Note that zero Lebesgue measure spectrum implies
absence of absolutely continuous spectrum,
and dimension of spectrum has some relation with quantum dynamics(\cite{L2,DT}).

When $v$ is periodic, the spectral property of $H_{\lambda,v}$ is well understood, it is known that $\sigma(H_{\lambda,v})$ is a union of finite intervals and has positive Lebesgue measure(see for example \cite{CL}).

When $v$ is less ordered, the situation is more complicated. Several classes of quasi-periodic potentials are extensively studied during the previous three decades.  One famous class is the so-called {\it almost Mathieu} potential, where $v^{am}(n)=\cos(n\alpha+\theta)$ with $\alpha\in \R\setminus \Q$ and $\theta\in \R.$ It is known that in this case $|\sigma(H_{\lambda,v^{am}})|=|4-2\lambda|$(\cite{L,JK,AK}). Another class is the potential generated by primitive substitution. It is shown in \cite{BG,LTWW,Lenz} that if $v$ is generated by a primitive substitution, then the spectrum $\sigma(H_{\lambda,v})$ has Lebesgue measure 0.

Among substitution class, the most famous one is the so-called  Fibonacci potential, which  is defined as
$$
v^{F}(n)=\chi_{(1-\alpha,1]}(n\alpha+\theta),
$$
where $\alpha=(\sqrt{5}-1)/2$ is the Golden number  and $\theta\in \R.$ We note  that when $\theta=0,$  $v^F$ can also be defined through the Famous Fibonacci substitution $\tau$: $\tau(a)=ab$ and $\tau(b)=a$(see for example \cite{Fo} Section 5.4).
The operator $H_{\lambda,v^{F}}$ is called {\it Fibonacci  Hamiltonian}. Since the pioneer works \cite{KKT,OPRSS}, Fibonacci Hamiltonian is always the central model in quasicrystal and is extensively studied, see for example the  survey \cite{DEG} and the most recent progress \cite{DGY}. 
The dimensional properties of $\sigma(H_{\lambda,v^F})$ has been well understood until now, see \cite{R, JL,LW,DEGT,C,DG,DG2}. In purticular  the following property is shown in \cite{DEGT}:
$$
\lim_{|\lambda|\to\infty} \dim_H \sigma(H_{\lambda,v^F})\ln  |\lambda|=\ln (1+\sqrt{2}).
$$
This implies that $\dim_H \sigma(H_{\lambda,v^F})\to 0$ with the speed $1/\ln |\lambda|$ when $|\lambda|\to\infty.$

Another famous potential is the so-called   Thue-Morse potential $w$, which  is defined   as follows:  Let $\sigma$ be the Thue-Morse sbustitution such that $\sigma(a)=ab$ and $\sigma(b)=ba$,  let $u=u_1u_2\cdots:=\sigma^\infty(a).$ For $n\ge 1$, let  $w(n)=1$ if $u_n=a$; let $w(n)=-1$
 if $u_n=b;$ let $w(1-n)=w(n)$ for $n\ge 1.$ The operator $H_{\lambda,w}$
   is called {\it Thue-Morse Hamiltonian}. Thue-Morse Hamiltonian is also studied by many authors, see \cite{AAKMP,MBNP,RSL,AP1,AP,Lu,DT} and so on.
However compared with Fibonacci case, almost noting is known about the dimensional properties of $\sigma(H_{\lambda,w})$, except some numerical result about the box dimension given in \cite{AP}.

In this paper, we will show the following

\begin{thm}\label{main-bd}
For  Thue-Morse potential  $w$  and any $\lambda\in\R$,
 $$
 \dim_H\sigma(H_{\lambda,w}) \ge \frac{\ln 2}{140\ln 2.1}.
 $$
 \end{thm}

 \begin{rem}
 {\rm
 1.
 Compare with the Fibonacci case, our result is a bit surprising since in Fibonacci case $\dim_H \sigma(H_{\lambda,v^F})$ will tends to 0 with speed $\ln |\lambda|$ when $|\lambda|\to\infty,$ however in Thue-Morse case there exists an absolute positive lower bound for $\dim_H \sigma(H_{\lambda,w})$
 for all $\lambda\in \R.$

 2. The lower bound is not optimal. It can be improved  through a  finer estimation. We do not pursue this since it does not give  the exact dimension.
 }
 \end{rem}

 In the following we say a few words about  the idea of the proof. Our method bases on the analysis of the behavior of the trace polynomials related to Thue-Morse Hamiltonian.   Some convergence behavior  hidden in these  polynomials enable us to construct a Cantor subset of the spectrum, meanwhile the Hausdorff dimension of the Cantor set can be estimated, which in turn offer a lower bound of the spectrum.

    Let us recall the definition of trace polynomials of Thue-Morse Hamiltonian, see \cite{AP,B} for more details and motivations about trace polynomials.
 Recall that $\sigma$ is the Thue-Morse substitution such that $\sigma(a)=ab$ and $\sigma(b)=ba.$ Denote the free group generated by $a,b$ as ${\rm FG}(a,b).$ Given $\lambda, x\in \R,$  define a homomorphism $\tau:{\rm
FG}(a,b)\to {\rm SL}(2,\R)$ as
$$
\tau(a)= \left[
\begin{array}{cc}
x-\lambda&-1\\
1&0
\end{array}
\right]\ \ \ \text{ and }\ \ \ \tau(b)= \left[
\begin{array}{cc}
x+\lambda&-1\\
1&0
\end{array}
\right]
$$
and $\tau(a_1\cdots a_n)=\tau(a_n)\cdots\tau(a_1).$   Define
$
h_n(x):={\rm tr}(\tau(\sigma^n(a)))
$
(where ${\rm tr}(A)$ denotes the trace of the matrix $A$),
then(\cite{AP,B})
\begin{equation}\label{h12}
h_1(x)=x^2-\lambda^2-2,\quad h_2(x)=(x^2-\lambda^2)^2-4x^2+2,
\end{equation}
and for $n\ge2$,
\begin{equation}\label{recurrence}
h_{n+1}(x)=h_{n-1}^2(x)(h_n(x)-2)+2.
\end{equation}
$\{h_n: n\ge 1\}$ is called the family  of {\it trace polynomials} related to Thue-Morse Hamiltonian.

 Define the zero set of the trace polynomials as
\begin{equation}\label{Sigma}
\Sigma=\{x\in\mathbb{R}: \exists n\ge 1,\ \  s.t.\ \ \  h_n(x)=0\}.
\end{equation}
It is shown in \cite{AP,B} that $\Sigma\subset \sigma(H_{\lambda,w})$. Since $\sigma(H_{\lambda,w})$ is closed, the closure of $\Sigma$ is also  a subset of the spectrum
(indeed  the closure of $\Sigma$ is exactly  the spectrum).
The set $\Sigma$ play a crucial role in our proof.
By the recurrence relation \eqref{recurrence} it is direct to check that, for any $x\in\Sigma$, if $h_n(x)=0$, then $h_{k}(x)=2$ for any $k\ge n+2$. Moreover, $x$ is a local maximal point of $h_k$ and the graph of $h_k$ is tangent to the horizontal line $y=2$ at point $(x,2).$  Figure \ref{attractor} shows the  typical configuration
when we plot the graphs of $\{h_n(x)\}$ around a point $x\in \Sigma$.

\begin{figure}[h]
\includegraphics[width=0.6\textwidth]{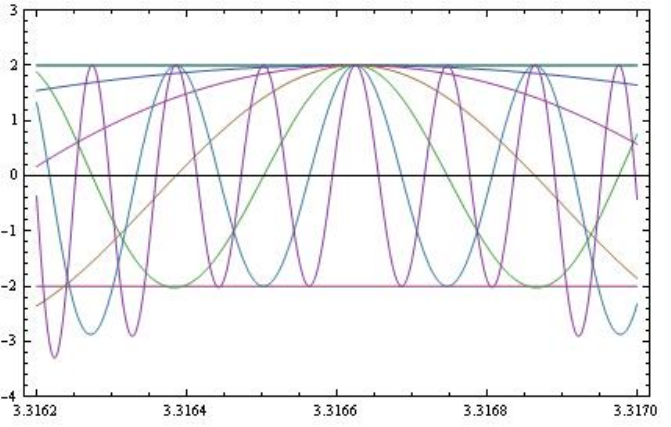}
\caption{Graph of $h_k(x)$, $k=2,\cdots,9$\ \ \  ($\lambda=3$).}\label{attractor}
\end{figure}

At first we describe a naive way of obtaining a lower bound for dimension of spectrum.
In Figure \ref{keygraph}, $a$ is such that $h_1(a)=0,$ then $h_4(a)=2.$ $h_4(x)$ is decreasing on the interval $[a,b]$ with $h_4(b)=0$,
$h_6(x)$ is decreasing   on $[a,c]$ with $h_6(c )=0$ and increasing on  $[d,b]$ with  $h_6(d)=0$. Consequently  $a,b,c,d\in \Sigma.$
Similarly when restricting to $[a,c]$ and drawing the graphs of $h_6$ and $h_8$, we can find two subintervals of $[a,c]$ such that the endpoints of these intervals are all in $\Sigma.$ For $[d,b]$ the situation is the same.
If we continue this process,  we can obtain  a covering structure which determines a Cantor set $E$ as its limit set.
Moreover by the construction all the endpoints of those intervals are in $\Sigma$. Since
 $E$ is the closure of all these endpoints, we conclude that $E$ is contained in the
closure of $\Sigma$, and hence contained in the spectrum.
The Hausdorff dimension of $E$ offers  a lower bound of
Hausdorff dimension of the spectrum.

\begin{figure}[h]
\includegraphics[width=1\textwidth]{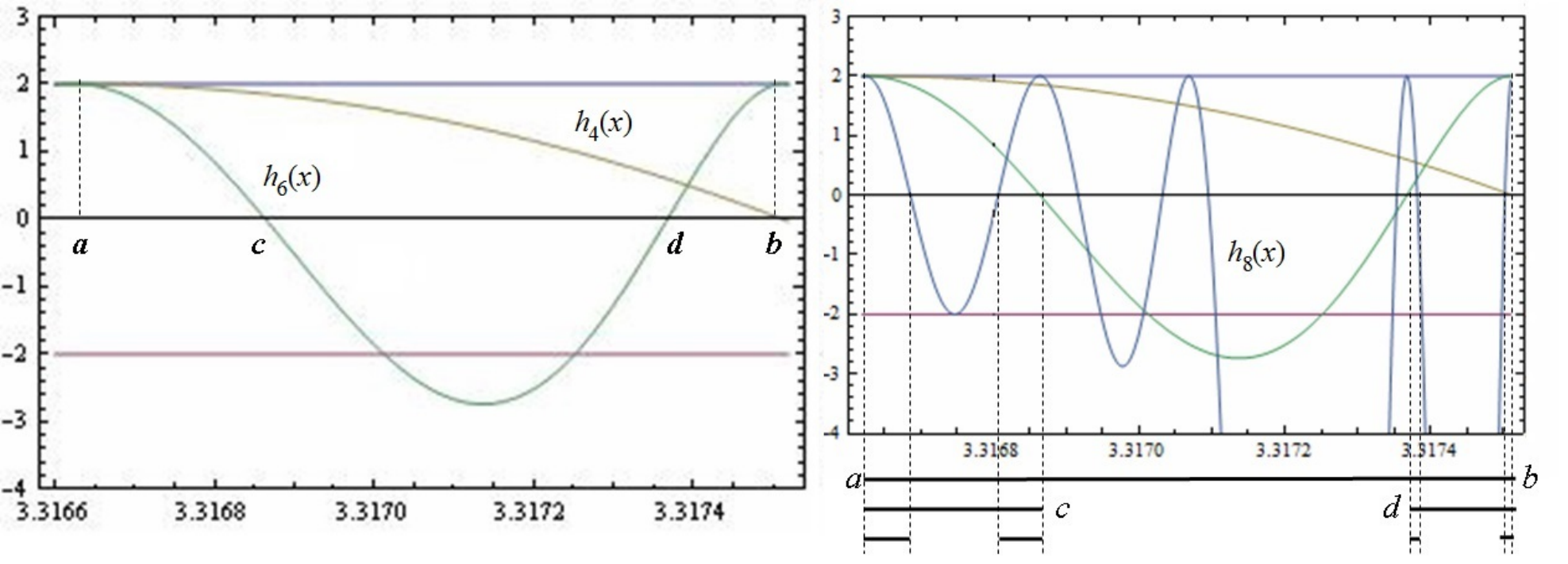}
\caption{The graph of $h_4(x),h_6(x),h_8(x)$ for $x\in(3.3166,3.3175)$.}\label{keygraph}
\end{figure}

To estimate the dimension of $E$, we need to estimate  ratios such as  $(c-a)/(b-a)$ and $(b-d)/(b-a)$.
However Figure \ref{keygraph} already suggests that the ratios may be out of control. We need more information about  trace polynomials and more delicate construction of subset.

A key observation is the following: fix any point $x_0$ in $\Sigma$  and  assume $h_n(x_0)=0$,
then there exists a scaling  factor $\rho$ such that
the rescaled sequence  $h_{n+3+k}(\frac{x}{2^k\rho}+x_0)$
will be more and more close to $2\cos x$ on any fixed interval $[-c,c]$ as $k$ tends to infinity(see again Figure \ref{attractor} for this phenomenon).  This closeness enable us to construct  a Cantor subset $E$ of $\sigma(H_{\lambda,w})$ in a controllable  way such that  the lower bound of $\dim_H E$ can be estimated explicitly.


More precisely we will start with   a polynomial pair $(P_{-1},P_0)$, which  have the following  expansions at $x_0:$
$$
 \begin{cases}
 P_{-1}(x)&=2-\frac{\rho^2}{4}(x-x_0)^2+O((x-x_0)^3);\\
  P_{0}(x)&=2-\rho^2(x-x_0)^2+O((x-x_0)^3).
 \end{cases}
$$

We can  define two kinds of closeness of $(P_{-1}, P_0)$ to $2\cos x$ at $x_0$: weak one and strong one(see Remark \ref{closeness}).
Then we  define a polynomial sequence $\{P_n:n\ge -1\}$ according to \eqref{recurrence}.

A crucial step is to establish the following inductive lemma (Lemma \ref{iterband}): There exists an absolute constant $K=140$ such that the following holds.
Assume $(P_{-1}, P_0)$ is strongly close to $2\cos x$ at $x_0$. Let  $y_0$ be  the minimal $y>x_0$ such that $P_0(y_0)=0$, then  $(P_{K-1},P_K)$ is strongly close to $2\cos x$ at both $x_0$ and $y_0$.
Moreover the following estimates hold:
$$\frac{y_K-x_0}{y_0-x_0}\ge 2.1^{-K}\ \ \ \text{ and }\ \ \  \frac{y_0-x_K}{y_0-x_0}\ge 2.1^{-K},$$
where $y_K$ is the minimal $y>x_0$ such that $P_K(y)=0$,
and $x_K$ is the maximal $y<y_0$ such that $P_K(y)=0$.

\begin{figure}[h]
\includegraphics[width=0.5\textwidth]{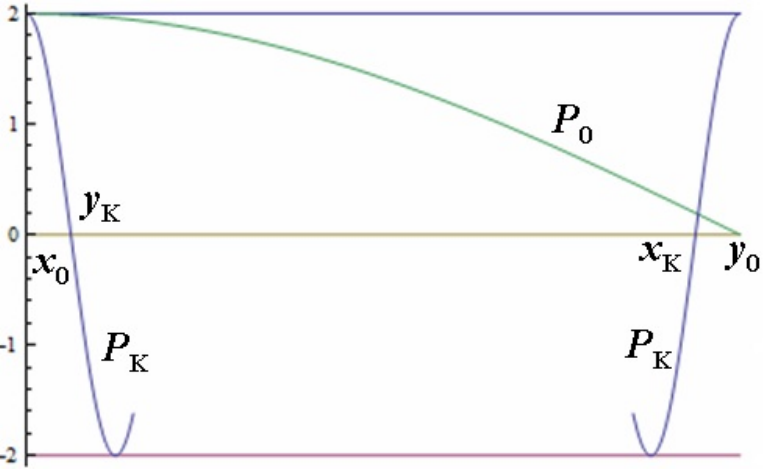}
\caption{Illustration of our process}\label{process}
\end{figure}

Note that this is the  right-side version of Lemma \ref{iterband}. If we define $y_0$ to  be  the maximal $y<x_0$ such that $P_0(y_0)=0$, then we can state the left-side version similarly.

Now for the pair $(P_{K-1},P_K)$, since it is strongly close to $2\cos x$ at both $x_0 $ and $y_0$, we can continue the process. Then inductively we can construct a Cantor set, for which we can estimate the  dimension.

Then what left is to find a trace polynomial pair such that it is indeed strongly close to $2\cos x$. We achieve this by two steps: at first we show that if $(P_{-1},P_0)$ is weakly close to $2\cos x$ at $x_0$, then $(P_{K-1},P_K)$ is strongly  close to $2\cos x$ at $x_0$(see Lemma \ref{basic})); next we show that $(h_4,h_5)$ is weakly close to $2\cos x$ at $a_\emptyset$, where $a_\emptyset$ is a zero of $h_1$(see Lemma \ref{initial-condition}). This will finish the proof of the main result.

 The rest of the paper is organized as follows. In Section \ref{sec-2}, we introduce the basic notations, state the main lemmas which are needed for the proof of main theorem. Then we finish the proof of Theorem \ref{main-bd}. In Section \ref{sec-3}, we prove Lemma \ref{initial-condition} and Lemma \ref{basic}. In Section \ref{inductive-lem}, we prove Lemma \ref{iterband}.


\section{Germ, closeness, proof of Theorem \ref{main-bd}}\label{sec-2}

In this section, at first we will introduce the notation of germ, which is the typical  configuration of the trace polynomial pairs $(h_{n-1},h_{n})$ around a zero point.  Next we will define the regularity of germ, which measure the closeness between the rescaled polynomial pair and $2\cos x$.  Then we will state several lemmas which describe the properties of regular germs. Base on these lemmas we will prove the main theorem.

 \subsection{Regular germ}\label{sec-21}\

 Given a polynomial pair  $(P_{-1}, P_0)$. Assume at  $x_0\in\mathbb{R}$,
 there exists $\rho>0$ such that
$$
 \begin{cases}
 P_{-1}(x)&=2-\frac{\rho^2}{4}(x-x_0)^2+O((x-x_0)^3);\\
  P_{0}(x)&=2-\rho^2(x-x_0)^2+O((x-x_0)^3)
 \end{cases}
$$
Then we say $(P_{-1},P_0)$ has a {\it  $\rho$-germ } at $x_0$.

We want to rescale  $P_{-1}$ and $P_0$ and compare them with $2\cos x.$ For this purpose,
for $k=-1$ and $0$ define
$$Q_k(x)=P_k(\frac{x}{2^k\rho}+x_0).
$$
It is ready to see  that
 $
 Q_{k}(x)=2-x^2+O(x^3).
 $
Since $2\cos x=2-x^2+O(x^3)$,  we have   $Q_k(x)=2\cos x +O(x^3).$
Write $\Delta_k(x)= Q_k(x)-2\cos x,$ then
$$\Delta_k(x)=Q_k(x)-2\cos x=\sum_{k\ge 3} \Delta_{k,n}x^n.
$$

We want to define a kind of smallness for $\Delta_k$ through its coefficients. Let us do some preparation.
Given two formal series with real coefficients
$$
f(x)=\sum_{n=0}^\infty a_nx^n \ \ \ \ \text{ and }\ \ \ \ g(x)=\sum_{n=0}^\infty b_nx^n,
$$
  define the following partial order:
$$
f\preceq g  \Leftrightarrow a_n\le b_n  \ \ (\forall n\ge 0).
$$
We further  define $|f(x)|^\ast:=\sum_{n=0}^\infty |a_n|x^n.$ Then it is easy to check that
$$
|fg|^\ast\preceq |f|^\ast|g|^\ast\ \ \ \text{ and }\ \ |f+g|^\ast\preceq |f|^\ast+|g|^\ast.
$$
Moreover if $|f|^\ast\preceq \tilde f$ and $|g|^\ast\preceq \tilde g$, then it is seen that $|fg|^\ast\preceq \tilde f\tilde g.$
Later we will use these properties repeatedly to estimate the coefficients of certain series.

Let us go back to $(P_{-1}, P_0)$.
 If moreover there exist $\delta>0$ and $\beta\ge1$ such that
$$
|\Delta_{-1}|^\ast, |\Delta_0|^\ast\preceq \delta\sum_{n=3}^\infty \frac{x^n}{\beta^n}
$$
Then we say that $(P_{-1},P_0)$  has a  {\it  $(\delta,\beta)$-regular $\rho$-germ at $x_0$}.  We also say that $(P_{-1},P_0)$  is  {\it  $(\delta,\beta)$-regular at $x_0$
with scaling  factor $\rho$}, or simply  {\it  $(\delta,\beta)$-regular at $x_0$}.
An immediate  observation is that if $\delta\le \delta^\prime$ and $\beta\ge \beta^\prime$ and $(P_{-1},P_0)$  is   $(\delta,\beta)$-regular at $x_0$, then $(P_{-1},P_0)$  is also   $(\delta^\prime,\beta^\prime)$-regular at $x_0$.

Recall that   $\{h_n(x):n\ge 1\}$ is the family of trace polynomials of Thue-Morse Hamiltonian
satisfy \eqref{h12} and \eqref{recurrence}. Let
\begin{equation}\label{start}
a_\emptyset=\sqrt{2+\lambda^2},
\end{equation}
then  $h_1(a_\emptyset)=0.$
The following lemma is the starting point of our whole proof.

\begin{lem}\label{initial-condition}
$(h_4,h_5)$ is $(1,1)$-regular at $a_\emptyset$.
\end{lem}

The following lemma shows that  the germ can keep when we iterating the pair. Moreover the regularity will become better and  better.

\begin{lem}\label{basic}
Assume $(P_{-1}, P_0)$ has a $\rho$-germ at $x_0$.
For $k\ge1$, define
\begin{equation}\label{iter}
P_k=P_{k-2}^2(P_{k-1}-2)+2.
\end{equation}
Then

1) $(P_{k-1},P_k)$ has a $2^k\rho$-germ at $x_0$ for any $k\ge 1.$

2) If $(P_{-1},P_0)$ is $(1,1)$-regular at $x_0$, then $(P_{k-1},P_k)$ is $(3200\cdot2^{-k/2},4)$-regular at $x_0$ for any $k\ge4.$

 \end{lem}

\subsection{Closeness and consequences }\

   Fix $\delta_0=10^{-2}, \delta_1=10^{-4}, \delta_2=10^{-16}$ and $K=140$. Then
\begin{equation}\label{coeff}
20\delta_1\le\delta_0; \quad
400^4\times 4\times 3\delta_2\le\delta_1; \quad 3200\cdot2^{-(K-4)/2}<\delta_2.
\end{equation}

\begin{rem}\label{closeness}
{\rm
Later in Section \ref{inductive-lem}, we will see that if $(P_{-1},P_0)$ is $(\delta,\beta)$-regular with $\delta$ small and $\beta$ big, then the rescaled  polynomials $Q_{-1}, Q_0$ will be close to $2\cos x$ in a bounded neighborhood of $0.$ In this sense $(\delta,\beta)$ give a measurement  of closeness between $(P_{-1},P_0)$ and $2\cos x$ around $x_0.$   We will   use the following three levels of closeness:
Assume $(P_{-1}, P_0)$ has a germ at $x_0$. If  $(P_{-1}, P_0)$ is $(1,1)$-regular (or $(\delta_1,2)$-regular, or $(\delta_2,4)$-regular) at $x_0$, we say that $(P_{-1}, P_0)$ is {\it weakly close ( close, or strongly close ) } to $2\cos x.$

}
\end{rem}

Our key technique lemma is the following:

\begin{lem}\label{iterband}
Let $(P_k)_{k\ge-1}$ satisfy the recurrence relation \eqref{iter}. Assume $(P_{-1},P_0)$ is $(\delta_2,4)$-regular at $x_0$.
For $k=0,K$, let $y_k^+$(resp. $y_k^-$) be the minimal $t>x_0$ (resp. maximal $t<x_0$) such that $P_k(t)=0$.
Let $x_K^+$ (resp. $x_K^-$) be the maximal $t<y_0^+$ (minimal $t>y_0^-$) such that $P_K(t)=0$. Then
$$\frac{|y_K^\pm-x_0|}{|y_0^\pm-x_0|}\ge 2.1^{-K},\quad
\frac{|y_0^\pm-x_K^\pm|}{|y_0^\pm-x_0|}\ge 2.1^{-K},$$
and $(P_{K-1},P_K)$ is $(\delta_2,4)$-regular at $x_0$ and $y_0^\pm$ respectively.
\end{lem}

See Figure \ref{process} for an illustration of the lemma. The crucial point is that, the strong closeness can pass to smaller  scale, which enable us to iterate the process.


\subsection{Proof of Theorem \ref{main-bd}}\

Now we will construct the desired Cantor set. To simplify the notation we write $\widetilde P_k(x):=h_{k+5}(x)$ for $k\ge -1.$ By Lemma \ref{initial-condition} we know that  $(\widetilde P_{-1},\widetilde P_0)$ is $(1,1)$-regular at $a_\emptyset.$

\begin{figure}[h]
\includegraphics[width=0.7\textwidth]{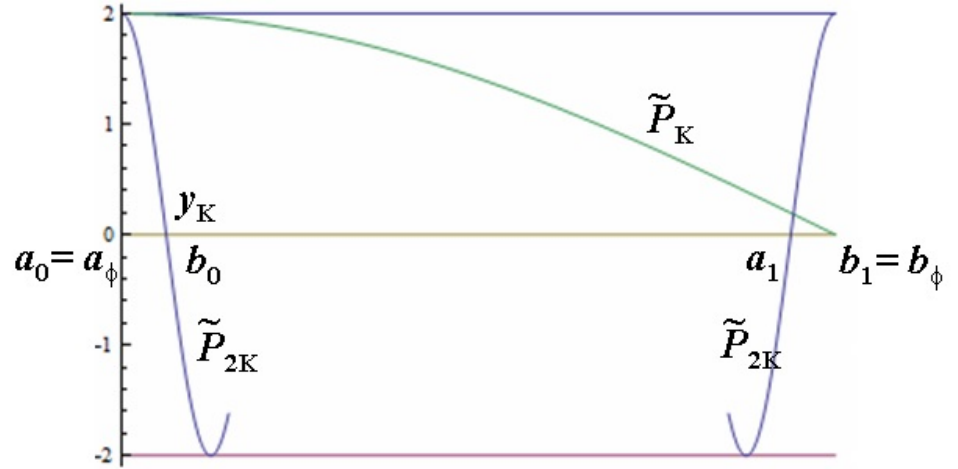}
\caption{Illustration of Cantor structure. \label{cantor} }
\end{figure}

We have fixed   $K=140$.
Assume $b_\emptyset>a_\emptyset$ is the first zero of $\widetilde P_{K}$ to
the right of $a_\emptyset.$ Define $I_\emptyset:=[a_\emptyset,b_\emptyset]$.
Assume $b_0$ is the smallest zero of $\widetilde P_{2K}$ in $I_\emptyset$ and $a_1$
is the biggest zero of $\widetilde P_{2K}$ in $I_\emptyset.$ Define
$$
I_0:=[a_\emptyset,b_0]=[a_0,b_0]\ \ \ \  \text{ and } \ \ \ I_1:=[a_1,b_\emptyset]=[a_1,b_1].
$$
Take any $w\in\{0,1\}^k$, suppose $I_w=[a_w,b_w]$ is defined.
 Assume $b_{w0}$ is the smallest zero of $\widetilde P_{(k+2)K}$ in $I_w$ and
 assume $a_{w1}$ is the biggest zero of $\widetilde P_{(k+2)K}$ in $I_w$.
 Write $a_{w0}=a_w$, $b_{w1}=b_w$ and define
 $$
 I_{w0}=[a_w,b_{w0}]=[a_{w0}, b_{w0}]\ \ \ \text{ and }\ \ \  I_{w1}
 =[a_{w1},b_{w}]=[a_{w1}, b_{w1}].
 $$
 Define  a Cantor set as
 $$
 \mathcal C:=\bigcap_{n\ge 0}\bigcup_{|w|=n}I_w
 $$

 \begin{prop}
  $$
 \dim_H \mathcal C \ge \frac{\ln 2}{K\ln 2.1}.
 $$
 \end{prop}

\begin{proof}
We claim  that
$$
\frac{|I_{w0}|}{|I_w|}\ge 2.1^{-K},\quad \frac{|I_{w1}|}{|I_w|}\ge 2.1^{-K}$$
and $(\widetilde P_{(k+2)K-1},\widetilde P_{(k+2)K})$ is $(\delta_2,4)$-regular at $a_w, b_w$
for  any $w\in \{0,1\}^*$. We show it by induction.

At first take $w=\emptyset$. Recall that $(\widetilde P_{-1},\widetilde P_0)$ is $(1,1)$-regular at $a_\emptyset$. By Lemma \ref{basic} 2) and \eqref{coeff},
    $(\widetilde P_{K-1},\widetilde P_{K})$ is $(\delta_2,4)$-regular at $a_\emptyset$.
By applying Lemma \ref{iterband} to the pair $(\widetilde P_{K-1},\widetilde P_K)$, we conclude that
$$\frac{|I_0|}{|I_\emptyset|}\ge 2.1^{-K},\quad \frac{|I_1|}{|I_\emptyset|}\ge 2.1^{-K}$$
and $(\widetilde P_{2K-1},\widetilde P_{2K})$ is  $(\delta_2,4)$-regular at both $a_\emptyset$ and   $b_\emptyset$.

Next take $k>0$ and assume the result holds for any $w\in \{0,1\}^{k-1}$.
Then  fix $w\in\{0,1\}^k$ and write $w=\tilde w j$ with $j=0$ or $1$.
If $w=\tilde w0$, then by the construction of ${\mathcal C},$ $b_w=b_{\tilde w0}$ is the smallest zero of $\widetilde P_{(k+1)K}$ in $I_{\tilde w}$. Moreover  by induction assumption, $(\widetilde P_{(k+1)K-1},\widetilde P_{(k+1)K})$ is $(\delta_2,4)$-regular at $a_{\tilde w}=a_w$.
By applying  Lemma \ref{iterband} to pair $(\widetilde P_{(k+1)K-1},\widetilde P_{(k+1)K})$ we conclude that
$$
\frac{|I_{w0}|}{|I_w|}\ge 2.1^{-K},\quad \frac{|I_{w1}|}{|I_w|}\ge 2.1^{-K}
$$
 and $(\widetilde P_{(k+2)K-1},\widetilde P_{(k+2)K})$ is  $(\delta_2,4)$-regular at  both $a_w$ and $b_w$.
If $w=\tilde w1$, then by the construction of ${\mathcal C},$ $a_w$ is the biggest zero of $\widetilde P_{(k+1)K}$ in $I_{\tilde w}$. Moreover  by induction assumption, $(\widetilde P_{(k+1)K-1},\widetilde P_{(k+1)K})$ is $(\delta_2,4)$-regular at $b_{\tilde w}=b_w$.
Again by applying  Lemma \ref{iterband},  the desired  result holds.

By induction,  the claim is proven.

 Now it is  well known  that (see for example \cite{F})
$$
\dim_H \mathcal C\ge \frac{\ln 2}{-\ln 2.1^{-K}}=\frac{\ln2}{K\ln 2.1}.
$$
\end{proof}

 \begin{proof}[Proof of Theorem \ref{main-bd}]
Recall that  $\Sigma$ is defined by \eqref{Sigma}.
By definition, for any $w\in\{0,1\}^*$, $a_w,b_w\in\Sigma\subset \sigma(H_{\lambda,w}).$
Since $\mathcal C=\overline{\{a_w,b_w: w\in \{0,1\}^* \}},$  we conclude that $\mathcal C \subset \sigma(H_{\lambda,w})$.
Consequently
 $$
 \dim_H\sigma(H_{\lambda,w}) \ge \dim_H \mathcal C \ge \frac{\ln 2}{K\ln 2.1}.
 $$
 Recall that  $K=140$ is an absolute positive constant, the result follows.
 \end{proof}


 \section{Proof of Lemma \ref{initial-condition} and \ref{basic}}\label{sec-3}

 In this section, at first we will show how a germ can appear. Next we show that $(h_4,h_5)$ has a $(1,1)$-regular germ at $a_\emptyset$ (Lemma \ref{initial-condition}).
 Then we will show that when iterate the polynomial pairs, the regularity will become better and better (Lemma \ref{basic}).

 \subsection{Generating  a  germ}\

 At first we present  a sufficient condition on how we can produce a germ when a  family of polynomials satisfies the recurrence relation \eqref{recurrence}.

 Given polynomial pair $(f_0,f_1)$. Define $f_{n+1}=f_{n-1}^2(f_n-2)+2$ for $n\ge 1.$

 \begin{lem} \label{suff-condi}
  Assume $f_0(x_0)=0$. If $f_1(x_0)< 2$ and $f_0^\prime(x_0), f_1(x_0)\ne 0$, then $(f_{k-1},f_k)$ has a germ at $x_0$ for $k\ge 4.$
 \end{lem}

 \proof\
 Write
 $$
 f_0(x)=f^\prime(x_0)(x-x_0)+O((x-x_0)^2)\ \text{ and }\ f_1(x)=f_1(x_0)+O((x-x_0)).
 $$
 By the recurrence relation we have
 \begin{eqnarray*}
 f_2(x)&=&2-(2-f_1(x_0))f_0^{\prime2}(x_0)(x-x_0)^2 +O((x-x_0)^3)\\
  f_k(x)&=&2-4^{k-3}(2-f_1(x_0))\left(f_0^{\prime}(x_0)f_1(x_0)\right)^2(x-x_0)^2 \\
  && +O((x-x_0)^3)\ \ \ \ (k\ge 3).
 \end{eqnarray*}

 If  $f_1(x_0)< 2$ and $f_0^\prime(x_0), f_1(x_0)\ne 0$, then
 $$
 \rho:= \sqrt{2-f_1(x_0)}|f_0^{\prime}(x_0)f_1(x_0)|>0
 $$ and   for $k\ge 3$
 \begin{equation}\label{fkx}
f_k(x)=2-(2^{k-3}\rho)^2(x-x_0)^2+O((x-x_0)^3).
\end{equation}
Then by the definition, $(f_{k-1},f_k)$ has a $2^{k-3}\rho$-germ at $x_0$ for all $k\ge 4.$
\hfill $\Box$

 Now we prove Lemma \ref{initial-condition}.

 \noindent {\bf  Proof of Lemma \ref{initial-condition}}\
By \eqref{h12},\eqref{recurrence} and \eqref{start},
$$h_1(a_\emptyset)=0,\ \
h_2(a_\emptyset)=-2-4\lambda^2<2, \ \  h_1^\prime(a_\emptyset)=2a_\emptyset \ne0,
\ \ h_2(a_\emptyset)\ne0.
$$
Write $\tau:=2^3(1+2\lambda^2)\sqrt{(1+\lambda^2)(2+\lambda^2)}$. By Lemma \ref{suff-condi},  $(h_4,h_5)$ has a $2\tau$-germ at $a_\emptyset$.

In the following we show that this germ is $(1,1)$-regular. Write $t:=2a_\emptyset$. Then $t\ge 2\sqrt{2}.$ Define
$$
g_n(x)=h_n(x+a_\emptyset).
$$
Then by direct computation,
$$
g_1(x)=t x+x^2\ \ \ \ \text{ and }\ \ \ \  g_2(x)=(6-t^2)+t^2 x^2+2t x^3+ x^4.
$$
Then we can compute that for $n\ge 4$,
$$
g_n(x)=2-4^{n-4}t^2(t^2-6)^2(t^2-4) x^2+O(x^3).
$$
By computation, $\tau= t(t^2-6)\sqrt{t^2-4}$. Define $f_n(x)=g_n(x/\tau)$. Then for $n\ge 4$ we have
$$
f_n(x)=2-4^{n-4}x^2+O(x^3).
$$
We also have
\begin{eqnarray*}
f_1(x)&=&{t}{\tau}^{-1}x+{\tau^{-2}}x^2\\
f_2(x)&=&(6-t^2)+(t/\tau)^2x^2+(2t/\tau^3)x^3+x^4/\tau^4\\
f_3(x)&=&2+O(x^2).
\end{eqnarray*}
By the fact that $t\ge 2\sqrt{2}$ and $\tau=t(t^2-6)\sqrt{t^2-4}$, it is direct to verify that
$$|f_1(x)|^*\preceq\frac{t}{\tau}xe^{x/32},\quad
|f_2(x)|^*\preceq(t^2-6)e^{x/4},\quad |f_2(x)-2|^*\preceq (t^2-4)e^{x/4}.
$$
Then, by $f_n=2+f_{n-2}^2(f_{n-1}-2)$, we have
\begin{equation}\label{f3x}
|f_3(x)|^*\preceq 2+(|f_1(x)|^{*})^2\cdot |f_2(x)-2|^\ast\preceq 2+\frac{1}{(t^2-6)^2}x^2e^{5x/16}.
\end{equation}
Since $f_3(x)=2+O(x^2)$, by \eqref{f3x} we have $|f_3(x)-2|^\ast\preceq {(t^2-6)}^{-2}x^2e^{5x/16}.$ Consequently
\begin{equation}\label{f4x}
|f_4(x)|^*\preceq 2+(|f_2(x)|^{*})^2\cdot |f_3(x)-2|^\ast\preceq 2+x^2e^{13x/16}.
\end{equation}
Since $f_4(x)=2-x^2+O(x^3)$, by \eqref{f4x} we have
$$
|f_4(x)-2+x^2|^\ast\preceq  \sum_{n\ge 3}\frac{x^n}{(n-2)!}.
$$
We also have
$$
|2\cos x-2+x^2|^\ast\preceq \sum_{n\ge 4}\frac{x^n}{n!}.
$$
Thus  we conclude that
\begin{equation}\label{1-1reg-4}
|f_4(x)-2\cos x|^*\preceq \sum_{n\ge3}x^n.
\end{equation}
Similarly since $f_4(x)=2+O(x^2)$, by \eqref{f4x} we have $|f_4(x)-2|^\ast\preceq x^2e^{13x/16}.$ Consequently
\begin{eqnarray*}
 |f_5(x)|^*&\preceq& 2+(|f_3(x)|^{*})^2\cdot |f_4(x)-2|^\ast\\
&\preceq& 2+4x^2e^{13x/16}+\frac{4x^4e^{18x/16}}{(t^2-6)^2}+\frac{x^6e^{23x/16}}{(t^2-6)^4}.
\end{eqnarray*}
Then we have
$$|f_5(x/2)|^\ast \preceq 2+x^2 e^{x/2}+\frac{x^4e^x}{16}+\frac{x^6e^x}{2^{10}}.$$
Since $f_5(x/2)=2-x^2+O(x^3)$ and $2\cos x=2-x^2+O(x^4),$
by similar argument as \eqref{1-1reg-4}, we can  show that
\begin{equation}\label{1-1reg-5}
|f_5(x/2)-2\cos x|^*\preceq \sum_{n\ge3}x^n.
\end{equation}
\eqref{1-1reg-4} and \eqref{1-1reg-5} implies that $(h_4,h_5)$ is $(1,1)$-regular at $a_\emptyset$
with scaling  factor $2\tau$.
\hfill $\Box$

 \subsection{Iteration of regular polynomial pairs}\

 In this subsection we will prove a strengthen  version of Lemma \ref{basic}, which will be needed in the proof of Lemma \ref{iterband}.

 Let us recall the setting in subsection \ref{sec-21}.
 Assume $(P_{-1}, P_0)$ has a $\rho$-germ at $x_0$.
For $k\ge1$, define $P_k$ by the recurrence relation \eqref{iter}.
Then it is direct to check that
\begin{equation}\label{expansion-p-k}
P_k(x)=2-4^k\rho^2(x-x_0)^2+O((x-x_0)^3), \ \  \ (\forall k\ge 1).
\end{equation}
For $k\ge -1$ define
\begin{equation}\label{iterscale}
Q_k(x)=P_k(\frac{x}{2^k\rho}+x_0).
\end{equation}
It is ready to show that
 $
 Q_{k}(x)=2-x^2+O(x^3).
 $
Since $2\cos x=2-x^2+O(x^3)$,  we conclude that  $Q_k(x)=2\cos x +O(x^3).$
Write $\Delta_k(x)= Q_k(x)-2\cos x,$ then
\begin{equation}\label{delta-k}
\Delta_k(x)=Q_k(x)-2\cos x=\sum_{k\ge 3} \Delta_{k,n}x^n.
\end{equation}
By the recurrence relation \eqref{iter}, we have for $k\ge 1$
\begin{eqnarray*}
Q_k(x)
&=&Q_{k-2}^2(x/4)(Q_{k-1}(x/2)-2)+2\\
&=&\Big(2\cos x/4+\Delta_{k-2}(x/4)\Big)^2\Big(2\cos x/2-2+\Delta_{k-1}(x/2)\Big)+2\\
&=&2\cos x+(2+2\cos \frac{x}{2})\cdot\Delta_{k-1}(\frac{x}{2})+
\\&&\Delta_{k-2}(\frac{x}{4})\Big(4\cos \frac{x}{4}+\Delta_{k-2}
(\frac{x}{4})\Big)\Big(2\cos \frac{x}{2}-2+\Delta_{k-1}(\frac{x}{2})\Big).
\end{eqnarray*}
Thus we conclude that for $k\ge 1$
\begin{eqnarray}\label{delta-k-x}
\Delta_k(x)&=&(2+2\cos \frac{x}{2})\cdot\Delta_{k-1}(\frac{x}{2})+\\
\nonumber&&\Delta_{k-2}(\frac{x}{4})\Big(4\cos \frac{x}{4}+\Delta_{k-2}
(\frac{x}{4})\Big)\Big(2\cos \frac{x}{2}-2+\Delta_{k-1}(\frac{x}{2})\Big).
\end{eqnarray}

The following proposition shows that how the coefficients evolves.

\begin{prop}\label{model}
Assume $\Phi_0,\Phi_1,\Phi_2$ are real analytic functions with  Taylor expension $\Phi_k(x)=\sum_{n\ge 3} \Phi_{k,n}x^n, \ k=0,1,2 $
 and satisfy the following relation:
\begin{eqnarray*}
\Phi_2(x)&=&(2+2\cos \frac{x}{2})\cdot\Phi_{1}(\frac{x}{2})+\\
\nonumber&&\Phi_{0}(\frac{x}{4})\Big(4\cos \frac{x}{4}+\Phi_{0}
(\frac{x}{4})\Big)\Big(2\cos \frac{x}{2}-2+\Phi_{1}(\frac{x}{2})\Big).
\end{eqnarray*}
If there exist $0<\delta\le1$ and $\beta\ge1$ such that
$$|\Phi_{0,n}|,  |\Phi_{1,n}|\le \delta \beta^{-n},\quad \forall n\ge3,$$
then when  $\beta=1$,
\begin{equation}\label{beta=1}
|\Phi_2(x)|^*\preceq
\delta\left(4\frac{x^3}{2^3}+4\frac{x^4}{2^4}+
9\sum\limits_{n\ge5}\frac{x^n}{2^n}\right).
\end{equation}
when  $\beta=2$,
\begin{equation}\label{beta=2}
|\Phi_2(x)|^*\preceq
\delta\left(4\frac{x^3}{4^3}+4\frac{x^4}{4^4}+24\frac{x^5}{4^5}+
43\sum\limits_{n\ge6}\frac{x^n}{4^n}\right).
\end{equation}
\end{prop}

\begin{proof}
 Write
$$
\begin{cases}
(I)&=(2+2\cos \frac{x}{2})\cdot\Phi_1(\frac{x}{2}),\\
(II)&=4\cos \frac{x}{4}+\Phi_0(\frac{x}{4}),\\
(III)&=2\cos \frac{x}{2}-2+\Phi_1(\frac{x}{2}).
\end{cases}
$$
Then
\begin{equation}\label{sum}
\Phi_2(x)=(I)+\Phi_{0}(\frac{x}{4})(II)(III),
 \end{equation}

We will frequently use the following facts:
\begin{equation}\label{use}
|\cos x|^*\preceq 1+\sum_{n\ge2}\frac{x^n}{n!},\quad
\sum_{k=2}^{n}\frac{\beta^k}{k!}<e^\beta-1-\beta,\quad
\sum_{k=3}^{n}2^{-k}<\frac{1}{4}.
\end{equation}
For example by the last two inequalities of \eqref{use} we have
\begin{equation}\label{two-useful}
\begin{array}{l}
\left(\sum\limits_{n\ge2}\frac{x^n}{n!2^n}\right)
\left(\sum\limits_{n\ge3}\frac{x^n}{(2\beta)^n}\right)=
\sum\limits_{n\ge5}\frac{x^n}{(2\beta)^n}\sum\limits_{k=2}^{n-3}\frac{\beta^k}{k!}\preceq
(e^\beta-1-\beta)\sum\limits_{n\ge5}\frac{x^n}{(2\beta)^n}\\
\left(\sum\limits_{n\ge3}\frac{x^n}{(4\beta)^n}\right)
\left(\sum\limits_{n\ge3}\frac{x^n}{(2\beta)^n}\right)=
\sum\limits_{n\ge6}\frac{x^n}{(2\beta)^n}\sum\limits_{k=3}^{n-3}2^{-k}\preceq
\frac{1}{4}\sum\limits_{n\ge6}\frac{x^n}{(2\beta)^n}
\end{array}
\end{equation}

By the assumption, \eqref{use} and  \eqref{two-useful}, we have
\begin{equation}\label{I}
\begin{array}{rcl}
|(I)|^*&\preceq&\left(4+2\sum\limits_{n\ge2}\frac{x^n}{n!2^n}\right)
\sum\limits_{n\ge3}\delta\frac{x^n}{(2\beta)^n}\\
&\preceq&4\delta\sum\limits_{n\ge3}\frac{x^n}{(2\beta)^n}+
2(e^\beta-1-\beta)\delta\sum\limits_{n\ge5}\frac{x^n}{(2\beta)^n}\\
&\preceq& 4\delta\sum_{n=3}^4\frac{x^n}{(2\beta)^n}+
2(e^\beta+1-\beta)\delta\sum\limits_{n\ge5}\frac{x^n}{(2\beta)^n}.
\end{array}\end{equation}

By $|(II)|^*\preceq4+4\sum\limits_{n\ge2}\frac{x^n}{n!4^n}+|\Phi_0(\frac{x}{4})|^*$
and $|\Phi_0(\frac{x}{4})|^*\preceq\delta\sum\limits_{n\ge3}\frac{x^n}{(4\beta)^n}$ we get
$$
\begin{array}{rcl}
|\Phi_0(\frac{x}{4})\times (II)|^*
&\preceq&4\delta\sum\limits_{n\ge3}\frac{x^n}{(4\beta)^n}+
4\delta\sum\limits_{n\ge5}\frac{x^n}{(4\beta)^n}\sum\limits_{k=2}^{n-3}\frac{\beta^k}{k!}+
(|\Phi_0(\frac{x}{4})|^*)^2\\
&\preceq&4\delta\sum\limits_{n\ge3}\frac{x^n}{(4\beta)^n}+
4(e^\beta-1-\beta)\delta\sum\limits_{n\ge5}\frac{x^n}{(4\beta)^n}+
(|\Phi_0(\frac{x}{4})|^*)^2\\
&\preceq&
4(e^\beta-\beta)\delta\sum\limits_{n\ge3}\frac{x^n}{(4\beta)^n}+
(|\Phi_0(\frac{x}{4})|^*)^2.
\end{array}$$
Since $|(III)|^*\preceq
2\sum\limits_{n\ge2}\frac{x^n}{2^n n!}+\delta\sum\limits_{n\ge3}\frac{x^n}{(2\beta)^n}$, by \eqref{two-useful} we have
\begin{eqnarray*}\allowdisplaybreaks
&&|\Phi_0(\frac{x}{4})\times (II)\times (III)|^*\\
&\preceq &
8(e^\beta-\beta)\delta\left(\sum_{n\ge3}\frac{x^n}{(4\beta)^n}\right)
\left(\sum_{n\ge2}\frac{x^n}{2^n n!}\right)\\
&&+4(e^\beta-\beta)\delta^2\left(\sum_{n\ge3}\frac{x^n}{(4\beta)^n}\right)
\left(\sum_{n\ge3}\frac{x^n}{(2\beta)^n}\right)\\
&&+2\delta^2\left(\sum_{n\ge3}\frac{x^n}{(4\beta)^n}\right)
\left(\sum_{n\ge3}\frac{x^n}{(4\beta)^n}\right)
\left(\sum_{n\ge2}\frac{x^n}{2^n n!}\right)\\
&&+\delta^3\left(\sum_{n\ge3}\frac{x^n}{(4\beta)^n}\right)
\left(\sum_{n\ge3}\frac{x^n}{(4\beta)^n}\right)
\left(\sum_{n\ge3}\frac{x^n}{(2\beta)^n}\right)\\
&\preceq &
(e^\beta-\beta)\delta\left(\sum_{n\ge3}\frac{x^n}{(2\beta)^n}\right)
\left(\sum_{n\ge2}\frac{x^n}{2^n n!}\right)\\
&&+4(e^\beta-\beta)\delta^2\left(\sum_{n\ge3}\frac{x^n}{(4\beta)^n}\right)
\left(\sum_{n\ge3}\frac{x^n}{(2\beta)^n}\right)\\
&&+2^{-2}\delta^2\left(\sum_{n\ge3}\frac{x^n}{(4\beta)^n}\right)
\left(\sum_{n\ge3}\frac{x^n}{(2\beta)^n}\right)
\left(\sum_{n\ge2}\frac{x^n}{2^n n!}\right)\\
&&+\delta^3\left(\sum_{n\ge3}\frac{x^n}{(4\beta)^n}\right)
\left(\sum_{n\ge3}\frac{x^n}{(4\beta)^n}\right)
\left(\sum_{n\ge3}\frac{x^n}{(2\beta)^n}\right)\\
&\preceq&
(e^\beta-\beta)\delta\sum\limits_{n\ge5}\frac{x^n}{(2\beta)^n}
\sum\limits_{k=2}^{n-3}\frac{\beta^k}{k!} \\
&&+(e^\beta-\beta)\delta^2\sum\limits_{n\ge6}\frac{x^n}{(2\beta)^n}\\
&&+2^{-2}(e^\beta-1-\beta)\delta^2\left(\sum_{n\ge3}\frac{x^n}{(4\beta)^n}\right)
\left(\sum\limits_{n\ge5}\frac{x^n}{(2\beta)^n}\right)\\
&&+\delta^3\left(\sum_{n\ge3}\frac{x^n}{(4\beta)^n}\right)
\left(\frac{1}{4}\sum\limits_{n\ge6}\frac{x^n}{(2\beta)^n}\right)\\
&\preceq&\frac{\beta^2(e^\beta-\beta)}{2}\delta\frac{x^5}{(2\beta)^5}+
(e^\beta-\beta)(e^\beta-1-\beta)\delta\sum\limits_{n\ge6}\frac{x^n}{(2\beta)^n}+\\
&&(e^\beta-\beta)\delta^2\sum\limits_{n\ge6}\frac{x^n}{(2\beta)^n}+
\frac{e^\beta-1-\beta}{16}\delta^2\sum\limits_{n\ge8}\frac{x^n}{(2\beta)^n}+
\frac{1}{16}\delta^3\sum\limits_{n\ge9}\frac{x^n}{(2\beta)^n},
\end{eqnarray*}
where, to get the first term of the last inequality, we notice that $\sum_{k=2}^{n-3}\frac{\beta^k}{k!}=\beta^2/2$ for $n=5$. 
Together with \eqref{I} and \eqref{sum} and $0<\delta\le1$, we get
$$\begin{array}{rcl}
|\Phi_2(x)|^*&\preceq&
4\delta\sum_{n=3}^4\frac{x^n}{(2\beta)^n}+
\left[(e^\beta-\beta)(2+\frac{\beta^2}{2})+2\right]\delta\frac{x^5}{(2\beta)^5}\\
&&+\left[(e^\beta-\beta)^2+\frac{33}{16}(e^\beta-\beta)+2\right]\delta\sum_{n\ge6}\frac{x^n}{(2\beta)^n}.
\end{array}
$$

By taking $\beta=1$ and $\beta=2$ respectively, we prove the proposition.
\end{proof}

 Now we show a strengthen version of Lemma \ref{basic}.

 \begin{lem}\label{control-coefficient}
Let $(P_k)_{k\ge-1}$ satisfy recurrence relation \eqref{iter} and $\delta\le1$. Assume $(P_{-1}, P_0)$ has a $\rho$-germ at $x_0$. Then

1) $(P_{k-1}, P_k)$ has a $2^k\rho$-germ at $x_0$ for any $k\ge 1$.

2) If $(P_{-1},P_0)$ is $(\delta,1)$-regular at $x_0$,
then $(P_{k-1},P_{k})$ is $(2\cdot {2}^{-k/2}\delta,1)$-regular at $x_0$ for any $k\ge1$,  is    $(36\cdot{2}^{-k/2}\delta,2)$-regular at $x_0$ for any $k\ge 2$ and  is  $(3200\cdot{2}^{-k/2}\delta,4)$-regular at $x_0$ for any $k\ge4$.

3) If $(P_{-1},P_0)$ is $(\delta,2)$-regular at $x_0$,
then for any $k\ge0$, $(P_{k-1},P_{k})$ is $(\delta,2)$-regular.
\end{lem}

\begin{proof}
1) follows from the definition of germ and \eqref{expansion-p-k}.

2)  Define $(Q_k)_{k\ge-1}$ and $(\Delta_k)_{k\ge-1}$
as in \eqref{iterscale} and \eqref{delta-k}. By \eqref{delta-k-x},  $\Delta_{k-1}, \Delta_k$ and $ \Delta_{k+1}$ satisfies the assumption of Proposition \ref{model} for $k\ge 0$.
 By the assumption of this lemma we have
$$
|\Delta_{-1}(x)|^*\preceq\delta\sum\limits_{n\ge3}x^n,\quad |\Delta_{0}(x)|^*\preceq \delta\sum\limits_{n\ge3}x^n.
$$
Then by \eqref{beta=1} and induction, for any $k\ge1$,
\begin{equation}\label{base1}
|\Delta_{2k-1}(x)|^*, |\Delta_{2k}(x)|^*\preceq
 {2}^{-k}\delta\sum\limits_{n\ge3}x^n.
\end{equation}
Consequently by the assumption, \eqref{base1}, \eqref{beta=1} and induction, for any $k\ge1$,
\begin{equation}\label{base2}
|\Delta_{2k-1}(x)|^*, |\Delta_{2k}(x)|^*\preceq 18\cdot {2}^{-k}\delta\sum\limits_{n\ge3}\frac{x^n}{2^n}.
\end{equation}
By \eqref{base2}, \eqref{beta=2} and induction,  for any $k\ge2$,
\begin{equation}\label{base3}
|\Delta_{2k-1}(x)|^*,|\Delta_{2k}(x)|^*\preceq 1600\cdot {2}^{-k}\delta\sum\limits_{n\ge3}\frac{x^n}{4^n}.
\end{equation}

\eqref{base1}  implies that $(P_{k-1},P_{k})$ is $(2\cdot {2}^{-k/2}\delta,1)$-regular at $x_0$ for any $k\ge1$. \eqref{base2}  implies that $(P_{k-1},P_{k})$ is $(36\cdot {2}^{-k/2}\delta,2)$-regular at $x_0$ for any $k\ge2$.  \eqref{base3}  implies that $(P_{k-1},P_{k})$ is $(3200\cdot {2}^{-k/2}\delta,4)$-regular at $x_0$ for any $k\ge4$.

3)
By the condition we have
$$
|\Delta_{-1}(x)|^*\preceq\delta\sum\limits_{n\ge3}\frac{x^n}{2^n},\quad
|\Delta_{0}(x)|^*\preceq \delta\sum\limits_{n\ge3}\frac{x^n}{2^n}.
$$
Then by  \eqref{beta=2} and induction, for any $k\ge1$,
$$
|\Delta_{k}(x)|^*\preceq \delta\sum\limits_{n\ge3}\frac{x^n}{2^n},
$$
which implies the result.
\end{proof}


 \section{Proof of Lemma \ref{iterband}}\label{inductive-lem}

 Recall that $\delta_0=10^{-2},  \delta_1=10^{-4}$ and $\delta_2=10^{-16}.$ We start with two simple geometric lemmas:

 \begin{lem}\label{zero}
 Assume $\varphi$ is a polynomial satisfying
$$
|\varphi(x)-2\cos x|\le \delta_0,\quad \forall x\in[0,1.9]\quad(\mbox{resp.} [-1.9,0]).
$$
We further assume $x_+$(resp. $x_-$)  is the minimal $x\in[0,\pi]$
(resp. the maximal $x\in[-\pi,0]$) such that $\varphi(x)=0$. Then
$$
|x_+-\pi/2|\le \delta_0,\quad (\mbox{resp.} |x_--(-\pi/2)|\le \delta_0).
$$
\end{lem}

\begin{proof}
By the assumption we have
$$
|2\cos x_+|=|\varphi(x_+)-2\cos x_+| \le \delta_0.
$$
On the other hand since $|\sin x|\ge 2|x|/\pi$ for $|x|\le \pi/2$ and $|\frac{\pi}{2}-x_+|\le \frac{\pi}{2}$
$$
|2\cos x_+|=2|\sin(\frac{\pi}{2}-x_+)|\ge \frac{4}{\pi}|x_+-\frac{\pi}{2}|.
$$
This prove the lemma for $x_+$. The proof for $x_-$ is analogous.
\end{proof}

 \begin{prop}\label{ratio}
Assume  polynomial pair $(P_{-1},P_0)$  is $(\delta_1,2)$-regular at $x_0$. Then
there exist $y^-_{-1}<y^-_0<x_0<y^+_0<y^+_{-1}$  such that, for $k=-1,0$,
$$
P_k(y^-_k)=P_k(y^+_k)=0; \ \  P_k(x)>0 \ \ (x\in I_k^-\cup I_k^+),
$$
where $I^-_k=(y^-_k,x_0]$, $I^+_k=[x_0,y^+_k)$.  Moreover
$$
\frac{|I^+_0|}{|I^+_{-1}|}, \ \ \ \frac{|I^-_0|}{|I^-_{-1}|}> 2.1^{-1}.
$$
\end{prop}

\begin{figure}[h]
\includegraphics[width=0.4\textwidth]{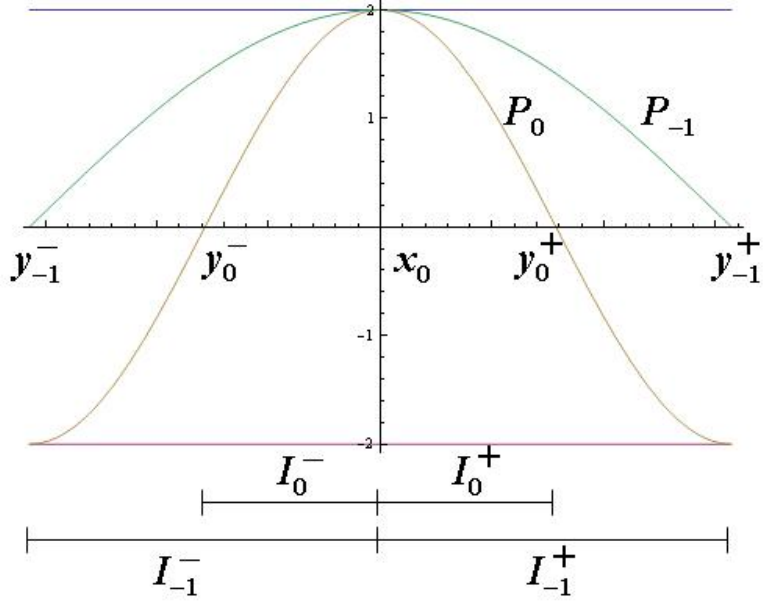}
\caption{Illustration of Proposition \ref{ratio}. }
\end{figure}

\proof\
Let $(P_{-1},P_0)$ be $(\delta_1,2)$-regular at $x_0$ with scaling  factor $a$.

For $k=-1,0$, define $Q_k(x)$ as in \eqref{iterscale}.
Then for $|x|<2$,
 $$
 |Q_{-1}(x)-2\cos x|, \ \ |Q_{0}(x)-2\cos x|\le \delta_1\sum_{n=3}^\infty \frac{|x|^n}{2^n}= \frac{|x|^3}{8-4|x|}\delta_1.
 $$
 Especially for $|x|\le 1.9$, we have
 $$
 |Q_{-1}(x)-2\cos x|, \ \ |Q_{0}(x)-2\cos x|\le \frac{1.9^3\delta_1}{8-4\cdot 1.9}<
 20\delta_1<\delta_0.$$

By Lemma \ref{zero},
$$|a(y_0^+-x_0)-\frac{\pi}{2}|<\delta_0, \quad |\frac{a}{2}(y_{-1}^+-x_0)-\frac{\pi}{2}|<\delta_0.$$
Then
 $$
 \frac{|I^+_{-1}|}{|I^+_0|}=\frac{|y_{-1}^+-x_0|}{|y_0^+-x_0|}\le
 \frac{2(\frac{\pi}{2}+0.01)}{\frac{\pi}{2}-0.01}<2.1.
 $$
The proof of the  other part of the proposition is analogous.
 \hfill $\Box$

 The proof of Lemma \ref{iterband} rely on the following technical  proposition.

\begin{prop}\label{branch2}
Let $(P_k)_{k\ge-1}$ satisfy recurrence  relation \eqref{iter}.
Assume  $(P_{-1},P_0)$ is
$(\delta_2,4)$-regular at $x_0$ with scaling  factor $a$.
Let  $y^+_0$( resp. $y^-_0$) be the minimal $t>x_0$ ( resp. maximal $t<x_0$)
such that $P_0(t)=0$.
Let  $x^+_3$( resp. $x^-_3$) be the maximal $t<y_0^+$ ( resp. minimal $t>y_0^-$)
such that $P_3(t)=0$.
Let $I_0^+=[x_0,y_0^+]$, $I_0^-=[y_0^-,x_0]$,  $I_3^+=[x_3^+,y_0^+]$ and $I_3^-=[y_0^-,x_3^-]$. Then
$$\frac{|I_3^+|}{|I_0^+|}\ge 2.1^{-3},\quad
\frac{|I_3^-|}{|I_0^-|}\ge 2.1^{-3}.$$

Moreover, $(P_3,P_4)$ is $(\delta_1,2)$-regular at both $y^{+}_0$ and $y^-_0$.
\end{prop}

\begin{figure}[h]
\includegraphics[width=0.4\textwidth]{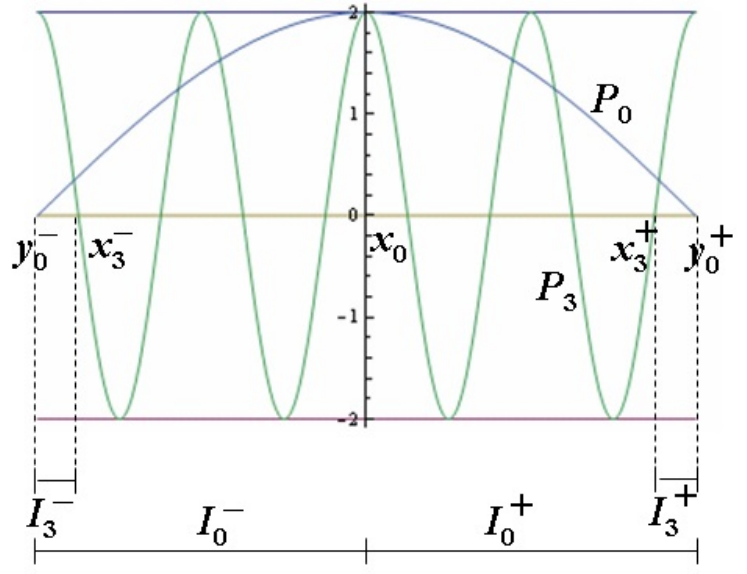}
\caption{Illustration of Proposition \ref{branch2}. }
\end{figure}

 \noindent {\bf Proof of Lemma \ref{iterband}}\
 We only consider the case in the interval $[x_0,y_0^+]$, aother case is  the same.

For $k=0,\cdots, K$, let $y_k$ be the minimal $t>x_0$ such that $P_k(t)=0$. Then $y_0=y_0^+$ and $y_K=y_K^+.$
For $k=3,\cdots, K$, let $x_k$ be the maximal $t<y_0^+$ such that $P_k(t)=0$. Then $x_K=x_K^+.$

$(P_{-1},P_0)$ is $(\delta_2,4)$-regular at $x_0$, then it is $(\delta_1,2)$-regular and $(1,1)$-regular at $x_0$ by the definition of regularity. By Lemma \ref{control-coefficient} 1), 2) and \eqref{coeff}, $(P_{K-1},P_K)$ is $(\delta_2,4)$-regular at $x_0$.

By  Lemma \ref{control-coefficient} 3), $(P_{k-1},P_k)$ is $(\delta_1,2)$-regular at $x_0$ for any $0\le k\le K$,
thus by Proposition \ref{ratio}, for any $0<k\le K$,
$$\frac{|y_k-x_0|}{|y_{k-1}-x_0|}\ge 2.1^{-1}.$$
And hence,
$$\frac{|y_K^+-x_0|}{|y_0^+-x_0|}=\frac{|y_K-x_0|}{|y_0-x_0|}\ge 2.1^{-K}.$$

By Proposition \ref{branch2}, $(P_3,P_4)$ is $(\delta_1,2)$-regular at $y_0^+=y_0$ and
\begin{equation}\label{three}
\frac{y_0-x_3}{y_0-x_0}\ge 2.1^{-3}.
\end{equation}
Same argument as above shows that  $(P_{K-1},P_K)$ is $(\delta_2,4)$-regular at $y_0$ and    $(P_{k-1},P_k)$ is $(\delta_1,2)$-regular at $y_0$ for any $3< k\le K.$
By Proposition \ref{ratio}, for any $3<k\le K$,
$$\frac{|y_0-x_k|}{|y_{0}-x_{k-1}|}\ge 2.1^{-1}.$$
Combining with \eqref{three} we conclude that
$$\frac{|y_0^+-x_K^+|}{|y_0^+-x_0|}=\frac{|y_0-x_K|}{|y_0-x_0|}\ge 2.1^{-K}.$$
This finish the  proof.
 \hfill $\Box$

 To prove Proposition \ref{branch2},  we  need the following  variant  of proposition \ref{model}.

 \begin{prop}\label{model2}
Assume $\Psi_0,\Psi_1,\Psi_2$ are real analytic functions with  Taylor expension $\Psi_k(x)=\sum_{n\ge 0} \Psi_{k,n}x^n, \ k=0,1,2 $
 and satisfy the following relation: for some $\tau\in \R$
\begin{eqnarray}\label{delta-k-x-1}
\Psi_2(x)&=&(2+2\cos \frac{x+\tau}{2})\cdot\Psi_{1}(\frac{x}{2})+\\
\nonumber&&\Psi_{0}(\frac{x}{4})\Big(4\cos \frac{x+\tau}{4}+\Psi_{0}
(\frac{x}{4})\Big)\Big(2\cos \frac{x+\tau}{2}-2+\Psi_{1}(\frac{x}{2})\Big).
\end{eqnarray}
If there exist $0<\delta\le1$ and $2\le\beta\le 3$ such that
$$|\Psi_{0,n}|,  |\Psi_{1,n}|\le \delta \beta^{-n},\quad \forall n\ge0,$$
then
\begin{equation*}
|\Psi_2(x)|^*\preceq
400\delta \sum\limits_{n\ge0}\frac{x^n}{\beta^n}.
\end{equation*}
\end{prop}

\begin{proof}
Write
$$\begin{array}{l}
(I)=\Big(2+2\cos (\frac{x}{2}+\frac{\tau}{2})\Big)\cdot \Psi_1(\frac{x}{2}),\\
(II)=4\cos (\frac{x}{4}+\frac{\tau}{4})+\Psi_0(\frac{x}{4}),\\
(III)=2\cos (\frac{x}{2}+\frac{\tau}{2})-2+ \Psi_1(\frac{x}{2}).
\end{array}
$$
Then
\begin{eqnarray}\label{Psi-2}
\Psi_2(x)=(I)+\Psi_0(\frac{x}{4})(II)(III).
\end{eqnarray}
Note that
\begin{equation}\label{cosine}
|\cos(x+x_0)|^*=|\sum_{n\ge0}\frac{\cos^{(n)}x_0}{n!}x^n|^\ast\preceq \sum_{n\ge0}\frac{x^n}{n!}.
\end{equation}
We have
$$
\begin{array}{rcl}
|(I)|^*&\preceq&\left(\sum\limits_{n\ge0}4\frac{x^n}{n!2^n}\right)\left(
\sum\limits_{n\ge0}\delta\frac{x^n}{2^n\beta^n}\right)
=
4\delta\sum\limits_{n\ge0}\frac{x^n}{(2\beta)^n}\sum_{k=0}\limits^{n}\frac{\beta^k}{k!}\\
&\preceq&
4\delta\sum\limits_{n\ge0}\frac{x^n}{\beta^n}\sum_{k=0}\limits^{n}\frac{(\frac{\beta}{2})^k}{k!}
\preceq 4 e^{\frac{\beta}{2}}\delta\sum\limits_{n\ge0}\frac{x^n}{\beta^n}
\end{array}$$
By \eqref{cosine} and the assumption we also have
$$
\begin{array}{rcl}
|\Psi_0(\frac{x}{4})\times (II)|^*&\preceq&\left(\delta\sum_{n\ge 0}\frac{x^n}{(4\beta)^n}\right)
\left(4\sum_{n\ge 0}\frac{x^n}{4^n n!}+\delta\sum_{n\ge 0}\frac{x^n}{(4\beta)^n}\right)\\
&\preceq&4e^{\frac{\beta}{2}}\delta\sum_{n\ge 0}\frac{x^n}{(2\beta)^n}+\delta^2\sum_{n\ge 0}\frac{x^n}{(4\beta)^n}(n+1)\\
&\preceq&4e^{\frac{\beta}{2}}\delta\sum_{n\ge 0}\frac{x^n}{(2\beta)^n}+\delta^2\sum_{n\ge 0}\frac{x^n}{(2\beta)^n}\\
&\preceq&\left(4e^{\frac{\beta}{2}}+\delta\right)\delta\sum_{n\ge 0}\frac{x^n}{(2\beta)^n}.
\end{array}
$$
Then we have
$$\begin{array}{rl}
&|\Psi_0(\frac{x}{4})\times (II)\times (III)|^*\\
\preceq&(4e^{\frac{\beta}{2}}+\delta)\delta\left(\sum\limits_{n\ge0}\frac{x^n}{(2\beta)^n}\right)
\left(4\sum\limits_{n\ge0}\frac{x^n}{2^n n!}+\delta\sum\limits_{n\ge0}\frac{x^n}{(2\beta)^n}\right)\\
=&(4e^{\frac{\beta}{2}}+\delta)\delta\left(4\sum\limits_{n\ge0}\frac{x^n}{\beta^n}\sum\limits_{k=0}^{n}\frac{(\frac{\beta}{2})^k}{k!}
+\delta\sum\limits_{n\ge0}\frac{x^n}{(2\beta)^n}(n+1) \right)\\
\preceq&(4e^{\frac{\beta}{2}}+\delta)^2\delta\sum\limits_{n\ge0}\frac{x^n}{\beta^n}.
\end{array}
$$
Since  $0<\delta\le1$ and $2\le\beta\le3$, by \eqref{Psi-2} we have
$$
|\Psi_2(x)|^*\preceq \left(4e^{\frac{\beta}{2}}+(4e^{\frac{\beta}{2}}+\delta)^2\right)\delta\sum_{n\ge0}\frac{x^n}{\beta^n}
\preceq 400\delta\sum_{n\ge0}\frac{x^n}{\beta^n}.
$$
This proves the proposition.
\end{proof}

\noindent {\bf Proof of Proposition \ref{branch2}.}\
  Recall that we have defined
  \begin{equation}\label{q-k}
Q_k(x)=P_k(\frac{x}{2^ka}+x_0)=:2\cos x+\Delta_k(x),\quad k=-1,0.
\end{equation}
Since $(P_{-1},P_0)$ is $(\delta_2,4)$-regular, we have
$$
|\Delta_{-1}(x)|^\ast \preceq \delta_2\sum_{n=3}^\infty \frac{x^n}{4^n},\quad
|\Delta_0(x)|^\ast \preceq \delta_2\sum_{n=3}^\infty \frac{x^n}{4^n}.
$$
And consequently, for $x\in(0,4) $ it is ready to show that for $k\ge 0$
 \begin{equation}\label{derivatives}
 \begin{array}{rcl}
 \begin{cases}
 |\Delta_0^{(k)}(x)|\\
  |\Delta_{-1}^{(k)}(x)|
 \end{cases}
 &\le& \delta_2 \left(\sum\limits_{n=3}^\infty\frac{x^n}{4^n}\right)^{(k)}=
 \frac{\delta_2}{16}\left(\frac{x^3}{4-x}\right)^{(k)}\\
 &=&\begin{cases}
 \frac{\delta_2}{16}\frac{x^3}{4-x}& k=0\\
 -\frac{\delta_2(x+2)}{8}+\frac{4\delta_2}{(4-x)^2}& k=1\\
-\frac{\delta_2}{8}+\frac{8\delta_2}{(4-x)^3} & k=2\\
 \frac{4 k!\delta_2}{(4-x)^{k+1}}& k\ge 3.
 \end{cases}
 \end{array}
 \end{equation}

$(P_{-1}, P_0)$ is $(\delta_2,4)$-regular implies that it is also $(\delta_2,2)$-regular.
Let $t_0$ be the minimal $t\in (0,2)$ such that $Q_0(t)=0$.  Then as proof in Proposition \ref{ratio} and Lemma \ref{zero},
\begin{equation}\label{est-t0}
P_0(t_0/a+x_0)=0,\quad |t_0-\frac{\pi}{2}|\le 20\delta_2.
\end{equation}
Consequently  $y^+_0=t_0/a+x_0$.  We have
$$
P_{-1}(y_0^+)=Q_{-1}(\frac{t_0}{2})=2\cos \frac{t_0}{2}+\Delta_{-1}(\frac{t_0}{2}).
$$
By \eqref{derivatives} and \eqref{est-t0} we get $|P_{-1}(y_0^+)-\sqrt{2}|\le 20\delta_2.$ Since
$$
P_{1}(y_0^+)=P_{-1}^2(y_0^+)(P_0(y_0^+)-2)+2,
$$
we conclude that
\begin{equation}\label{P1}
|P_1(y_0^+)+2|\le 120\delta_2.
\end{equation}
By \eqref{q-k}, $P_0^\prime(y_0^+)=aQ_0^\prime(t_0)=a(-2\sin t_0+\Delta_0^\prime(t_0))$. By \eqref{derivatives} and \eqref{est-t0},
\begin{equation}\label{P0}
\left|\frac{P_0^\prime(y_0^+)}{a}+2\right|\le 2\delta_2.
\end{equation}
By \eqref{fkx} and \eqref{P1}, \eqref{P0}, for $k=3,4$ we have
$$
P_k(x)=2-(2^{k-3}\rho)^2(x-y_0^+)^2 +O((x-y_0^+)^3)
$$
with $\rho=\sqrt{2-P_1(y_0^+)}|P_0^\prime (y_0^+)P_1(y_0^+)|$.
Thus
$$
\frac{\rho}{a}=\sqrt{4+\varepsilon_1}(2+\varepsilon_2)(2+\varepsilon_1)
$$
with $|\varepsilon_1|<120\delta_2$ and  $|\varepsilon_2|<2\delta_2$. Consequently we have
$$\left|\frac{\rho}{8a}-1\right|<100\delta_2,\quad \left|\frac{8a}{\rho}-1\right|<100\delta_2.$$
For $k\ge 3$, if we define $\tilde Q_k(x):= P_k(\frac{x}{2^{k-3}\rho}+y_0^+),$ then
\begin{equation}\label{tildeqk}
\tilde Q_k(x)=2-x^2+O(x^3)=2\cos x+O(x^3)=: 2\cos x +\tilde \Delta_k(x).
\end{equation}
We need to show that $(P_3,P_4)$ is $(\delta_1,2)$-regular,  i.e.,
$$
|\tilde\Delta_3(x)|^\ast, |\tilde\Delta_4(x)|^\ast \preceq \delta_1\sum_{n\ge 3} \frac{x^n}{2^n}.
$$

To get this result, we study $\bar\Delta_k(x):= \Delta_k(x+2^k t_0))$ first. We have
\begin{equation}\label{pk-1}
\begin{array}{rcl}
 P_k(\frac{x}{a}+y_0^+)&=&P_k(\frac{x+t_0}{a}+x_0)\ =\ Q_k(2^k(x+t_0))\\
&=&2\cos(2^k(x+t_0))+\Delta_k(2^k (x+t_0)\\
&=& 2\cos(2^k(x+t_0))+\bar\Delta_k(2^k x).
 \end{array}
\end{equation}
By the recurrence relation of $P_k$, it is ready to check  that
$$
\Psi_0: =\bar \Delta_{k-2}, \Psi_1: =\bar \Delta_{k-1}, \Psi_2: =\bar \Delta_{k}
$$ satisfies \eqref{delta-k-x-1} with $\tau=2^kt_0$ for $k\ge 1$.

On the other hand  we have
 $$
 \begin{cases}
 \bar \Delta_0(x)&=\Delta_0(x+t_0)=\sum_{n=0}^\infty \frac{\Delta_0^{(n)}(t_0)}{n!}x^n\\
 \bar \Delta_{-1}(x)&=\Delta_{-1}(x+t_0/2)=\sum_{n=0}^\infty \frac{\Delta_{-1}^{(n)}(t_0/2)}{n!}x^n.
 \end{cases}
$$
Write $\beta=4-\pi/2-0.01=2.419\cdots$ By \eqref{derivatives} and \eqref{est-t0} we get
$$
 |\bar\Delta_0(x)|^\ast,  |\bar \Delta_{-1}(x)|^\ast\preceq 4\delta_2\sum_{n=0}^\infty \frac{x^n}{\beta^n}.
$$
Recall that by \eqref{coeff} we have   $4\cdot 400^4\delta_2\le\delta_1/3$. By Proposition \ref{model2} and induction,  for $k=3,4$ we get
\begin{equation}\label{bar-Delta-k}
|\bar\Delta_k(x)|^\ast\preceq 4\cdot 400^k \delta_2\sum_{n=0}^\infty \frac{x^n}{\beta^n}
\preceq \frac{\delta_1}{3}\sum_{n=0}^\infty \frac{x^n}{\beta^n}.
\end{equation}

By \eqref{tildeqk} and \eqref{pk-1},  for $k=3,4$ we have
$$\begin{array}{rcl}
\tilde\Delta_k(x)&=&\tilde Q_k(x)-2\cos x\\
&=&P_k(\frac{x}{2^{k-3}\rho}+y_0^+)-2\cos x\\
&=&\bar\Delta_k(\frac{8a}{\rho}x)+2\cos(\frac{8a}{\rho}x+2^k t_0)-2\cos x\\
&=&\bar\Delta_k(\frac{8a}{\rho}x)-2\sin2^k t_0\sin(\frac{8a}{\rho}x)+2\cos2^kt_0\cos(\frac{8a}{\rho}x)-2\cos x.
\end{array}
$$
Notice that  $|\frac{8a}{\rho}-1|<100\delta_2=10^{-14}$, then by \eqref{bar-Delta-k}, for $k=3,4$ we have
\begin{equation}\label{EST-1}
|\bar\Delta_k(\frac{8a}{\rho}x)|^*\preceq \frac{\delta_1}{3}\sum_{n=0}^\infty \frac{x^n}{2^n}.
\end{equation}
Moreover,
$|t_0-\frac{\pi}{2}|<20\delta_2$ implies for $k=3,4$,
$$|\sin 2^kt_0|<320\delta_2,\quad |1-\cos 2^kt_0|<320\delta_2. $$
Consequently
\begin{equation}\label{EST-2}
|2\sin2^k t_0\sin(\frac{8a}{\rho}x)|^*\preceq640\delta_2\sum\limits_{n\ge1}\frac{(\frac{8a}{\rho}x)^{2n-1}}{(2n-1)!}
\preceq \frac{\delta_1}{3} \sum_{n=0}^\infty \frac{x^n}{2^n}.
\end{equation}
Finally we have
\begin{equation}\label{EST-3}
|2\cos2^kt_0\cos(\frac{8a}{\rho}x)-2\cos x|^*\preceq
2 \sum\limits_{n\ge0}\frac{|\cos2^kt_0(\frac{8a}{\rho})^{2n}-1| x^{2n}}{(2n)!}\\
\preceq \frac{\delta_1}{3} \sum_{n=0}^\infty \frac{x^n}{2^n},
\end{equation}
where we have used the fact that for $n\le10$,
$$|\cos2^kt_0(\frac{8a}{\rho})^{2n}-1|\le |\cos2^kt_0-1|+|(\frac{8a}{\rho})^{2n}-1|<\delta_1/16;$$
for $n>10$,
$$\frac{2|\cos2^kt_0(\frac{8a}{\rho})^{2n}-1|}{(2n)!}<\frac{2+2(1+100\delta_2)^{2n}}{(2n)!}<
\frac{\delta_1}{3\times 2^{2n}}.$$
Combine \eqref{EST-1}, \eqref{EST-2} and \eqref{EST-3}, we conclude that   for $k=3,4$
$$|\tilde\Delta_k(x)|^* = |\tilde Q_k(x)-2\cos x|\preceq \delta_1\sum_{n\ge3}\frac{x^n}{2^n}.$$
This proves that $(P_3,P_4)$ is $(\delta_1,2)$-regular at $y_0^+$ with scaling factor $2\rho$.

Since $(P_3,P_4)$ is $(\delta_1,2)$-regular at $y_0^+$, analogous to the proof of Proposition \ref{ratio}, we have
$$|\rho(y_0^+-x_3^+)-\frac{\pi}{2}|<\delta_0.$$
This implies $|\rho|I_3^+|-\frac{\pi}{2}|<0.01$.

On the other hand since  $(P_{-1},P_0)$ is $(\delta_1,2)$-regular at $x_0$ with scaling factor $a$, we have
$$|a(y_0^+-x_0)-\frac{\pi}{2}|<\delta_0.$$
This implies $|a|I_0^+|-\frac{\pi}{2}|<\delta_0$. Hence
$$\frac{|I_3^+|}{|I_0^+|}\ge \frac{1}{8}\frac{8a}{\rho}\frac{\frac{\pi}{2}-0.01}{\frac{\pi}{2}+0.01}\ge2.1^{-3}.$$

This prove the result for $y_0^+$. The proof for $y_0^-$ is the same.
\hfill $\Box$

\medskip

\noindent
{\bf Acknowledgements}.  The authors thank Professor  Jean Bellissard for many valuable  suggestions. Liu and Qu are supported by the National Natural Science Foundation of China, No. 11371055.  Qu is supported by the National Natural Science Foundation of China, No. 11201256.


\end{document}